\documentclass{amsart}

\usepackage{mathtools}

\newtheorem{theorem}{Theorem}[section]
\newtheorem{lemma}[theorem]{Lemma}

\theoremstyle{definition}
\newtheorem{definition}[theorem]{Definition}
\newtheorem{example}[theorem]{Example}

\newtheorem{proposition}[theorem]{Proposition}
\newtheorem{corollary}[theorem]{Corollary}

\theoremstyle{remark}
\newtheorem{remark}[theorem]{Remark}

\numberwithin{equation}{section}

\begin{document}
	\title[product of $k$ elements in a finite ring is zero]{The probability that the product of $k$ elements in a finite ring is zero}

	\author{Dibyasman Sarma}
	\address{Department of Mathematics, National Institute of Technology Meghalaya\\ Sohra, Meghalaya
		793108, India}
	\email{dibyasman.sarma@gmail.com}

	\author{Tikaram Subedi}
	\address{Department of Mathematics, National Institute of Technology Meghalaya\\ Sohra, Meghalaya
		793108, India}
	\email{tsubedi2010@gmail.com}
	
	
	
	
	\begin{abstract}
		In this paper, for a fixed integer $k\ge 2$, we study the probability that the  product of $k$ randomly chosen elements in a finite commutative ring $R$ is zero, which we denote by $zp_{_k}(R)$. We investigate bounds for $zp_{_k}(R)$ that turn out to be sharp bounds for certain classes of rings. Further, we determine the maximum value of $zp_{_k}(R)$ that can be obtained for any ring $R$, and  classify all rings within some specific range of $zp_{_k}(R)$.
	\end{abstract}
	
	\maketitle
	
\section{Introduction}
Providing numerical measurements on how close an algebraic structure is from satisfying certain properties has been of interests to several mathematicians over the last several years. In this regard, mathematicians have studied the concept of probability that a finite algebraic structure satisfies certain properties. Erd\"os and Tur\'an \cite{erdos} started the study of commutativity in finite groups by using statistical approaches and their work showed continuous extensions. In recent years, the commutativity of finite rings has been studied via probability (for example see, \cite{buckley1,buckley2,buckley3,dutta, machale}). Also, the probability that the product of two randomly chosen elements in a finite ring have product zero called the nullity degree was studied in \cite{emskhani1} and continued in \cite{dolzan, emskhani2}. The nullity degree of a ring gives a measure of how close the ring multiplication is to being trivial multiplication. Naturally, the  question concerning if the concept of nullity degree can be extended to any number of elements arises.

\medspace

In this paper, for any ring $R$ with identity and integer $k\ge2$, we investigate the probability that the product of  $k$ randomly chosen elements in $R$ is zero, which we denote by $zp_{_k}(R)$. Thus, 
\[zp_{_k}(R) =\frac{\left|\left\{(x_1,x_2,...,x_k)\in R^k:x_1x_2...x_k=0\right\}\right|}{|R|^k}.\]
First, we obtain an upper and a lower bounds for $zp_{_k}(R)$ that are recursive and depend on the structure of $R$, and then classify rings for which the same turn out to be sharp bounds.  Subsequently, we obtain an explicit form of the upper bound for $zp_{_k}(R)$. Finally, we acquire the maximum value of $zp_{_k}(R)$ obtained for any ring $R$, and determine the structure of all rings within a certain range of $zp_{_k}(R)$.

\medspace

Throughout this paper, $R$ is a finite commutative ring with identity and $Z(R)$ is the set of all $zero ~divisors$ in $R$. For a set $S$, $|S|$ is the $cardinality$ of $S$. For any $x\in R$, $Ann(x)$ is the $annihilator$ of $x$ in $R$, which is the set $\{y\in R : xy=0\}$. By $(R, \mathcal{M})$, we denote the $local$ ring $R$ with its maximal ideal $\mathcal{M}$. For integers $n,q\ge2$, $\mathbb{Z}_n$ and $GF(q)$ are the ring of integers modulo $n$ and finite field of size $q$, respectively.

\section{Bounds for $zp_{_k}(R)$}
\begin{definition}
	For any ring $R$ and integer $k\ge 1$, we define
	\[Ann_{_{k}}(R)=\left\{(x_1,x_2,...,x_k)\in R^k:x_1x_2...x_k=0\right\}.\]
\end{definition}

In the following, we establish an upper bound and a lower bound for  $zp_{_k}(R)$ for any ring $R$ and integer $k\ge 2$. 

\begin{theorem} \label{t1}
	For any ring $R$ of order $n$ and integer $k\ge 2$, we have
	\begin{align*}
		(1)~zp_{{_k}}(R)&\ge \frac{\splitfrac{n^{k-1}+(n-|Z(R)|) |Ann_{{_{k-1}}}(R)|+\displaystyle\sum_{x(\ne 0)\in Z(R)}\left\{n^{k-1} \right.}{\left.-(n-|Ann(x)|)^{k-1}\right\}}}{n^k},\\
		&\\
		(2)~zp_{{_k}}(R)&\le \frac{\splitfrac{n^{k-1}+(n-|Z(R)|) |Ann_{{_{k-1}}}(R)|+\displaystyle\sum_{x(\ne 0)\in Z(R)}\left\{n^{k-1}\right.}{\left.- \big(n +(k-2)|Z(R)|-(k-1)|Ann(x)|\big)(n-|Z(R)|)^{k-2}\right\}}}{n^k}.
	\end{align*}	
\end{theorem}
\begin{proof}
	We consider the following three sets:
	\vspace{-.2cm}
	\begin{align*}
		& A_k=\left\{(x_1,x_2,...,x_k)\in Ann_{_{k}}(R):x_1=0\right\},\\
		& B_k=\left\{(x_1,x_2,...,x_k)\in Ann_{_{k}}(R):x_1\notin Z(R)\right\}~\text{and}\\
		& C_k=\left\{(x_1,x_2,...,x_k)\in Ann_{_{k}}(R):x_1(\ne 0)\in Z(R)\right\}.
	\end{align*}
	Then $|Ann_{_{k}}(R)|=|A_k|+|B_k|+|C_k|$. Clearly, $|A_k|=n^{k-1}$ and $|B_k|=(n-|Z(R)|)|Ann_{_{k-1}}(R)|$. To estimate $|C_k|$, first we define $C_k(x_1)=\{(x_2,x_3,...,$ $x_k)\in R^{k-1} : x_2x_3...x_k\in Ann(x_1)\}$ for fixed $x_1(\ne0)\in Z(R)$. It is obvious that $ |C_k|=\sum\limits_{x_1(\ne0)\in Z(R)}|C_k(x_1)|$. Again, we get
	\begin{align}
		C_k^{'}(x_1)\subseteq C_k(x_1)&\subseteq C_k^{'}(x_1) \cup \left( \left(R\setminus Ann(x_1)\right)^{k-1}\setminus C_k^{''}(x_1)\right) \label{e2},
	\end{align}
	where
	\begin{align} 
		C_k^{'}(x_1)=\left[ Ann(x_1)\times R^{k-2}\right]\cup 	\left[(R\setminus Ann(x_1))\times Ann(x_1)\times R^{k-3}\right] \label{e3}\\ \cup \left[(R\setminus Ann(x_1))^2 \times Ann(x_1)\times R^{k-4}\right]\cup ...\nonumber \\ \cup \left[(R\setminus Ann(x_1))^{k-2}\times Ann(x_1)\right] \nonumber
	\end{align}
	and
	\begin{align} 
		C_k^{''}(x_1)=&\left[(Z(R)\setminus Ann(x_1))\times 	(R\setminus Z(R))^{k-2}\right]\label{e4}\\&\hspace{1.1cm}\cup \left[(R\setminus Z(R)) 
		\times (Z(R)\setminus Ann(x_1))\times (R\setminus Z(R))^{k-3}\right]\cup ...\nonumber \\ &\hspace{1.9cm}\cup \left[ (R\setminus Z(R))^{k-3} \times (Z(R)\setminus Ann(x_1))\times (R\setminus Z(R))\right]\nonumber\\ &\hspace{4.4cm}\cup \left[(R\setminus Z(R))^{k-2}\times (R\setminus Ann(x_1))\right].\nonumber
	\end{align}
	Observe that (\ref{e2}) follows from the fact that $C_k^{''}(x_1)\subseteq (R\setminus Ann(x_1))^{k-1}$ and for any $(x_2,x_3,...,x_k)\in C_k^{''}(x_1)$, we have $x_1x_2...x_k\ne 0$. Also, note that the sets on the right sides of both (\ref{e3}) and (\ref{e4}) are disjoint. Thus,
	\begin{align*}
		|C_k(x_1)|&\ge |Ann(x_1)|n^{k-2} + (n-|Ann(x_1)|)|Ann(x_1)|n^{k-3} 	+ ...\\&\hspace{6.3cm}+ (n-|Ann(x_1)|)^{k-2}|Ann(x_1)|\\
		&=n^{k-1}-(n-|Ann(x_1)|)^{k-1}\\
		\text{and}\hspace{.67cm} &\\
		|C_k(x_1)| & \le n^{k-1}-(n-|Ann(x_1)|)^{k-1} + \left[(n-|Ann(x_1)|)^{k-1}-\left\{(k-2)(|Z(R)|\right.\right.\\ &\hspace{1cm}\left.-|Ann(x_1)|)
		(n-|Z(R)|)^{k-2} + (n-|Ann(x_1)|)(n-|Z(R)|)^{k-2}\big\}\right]\\
		& = n^{k-1} - \{n+(k-2)|Z(R)|-(k-1)|Ann(x_1)|\}(n-|Z(R)|)^{k-2}.
	\end{align*}
	Therefore,

	\begin{align*}
		|C_k|&\ge\sum_{x(\ne 0)\in 	Z(R)}\left\{n^{k-1}-(n-|Ann(x)|)^{k-1}\right\},\\
		|C_k|&\le\sum_{x(\ne 0)\in Z(R)}\left[n^{k-1} - \{n+(k-2)|Z(R)|-(k-1)|Ann(x)|\}(n-|Z(R)|)^{k-2}\right],
	\end{align*}
	from which the desired inequalities follow.
\end{proof}

\begin{corollary}
	Let $R$ be a ring of order $n$, where $Ann(x)$ is a prime ideal for each $x(\ne0)\in Z(R)$. Then for any $k\ge 2$, we have
	\[zp_{_k}(R)=\frac{\splitfrac{n^{k-1}+(n-|Z(R)|) |Ann_{{_{k-1}}}(R)|+\displaystyle\sum_{x(\ne 0)\in Z(R)}\big\{n^{k-1}}{-(n-|Ann(x)|)^{k-1}\big\}}}{n^k}.\]
\end{corollary}
\begin{proof}
	The proof follows from the fact that if $Ann(x)$ is a prime ideal for a non-zero zero divisor $x$, then from (\ref{e2}) in Theorem \ref{t1}, we have $C_k^{'}(x)=C_k(x)$, whence the required result holds.
\end{proof}

The following corollary is a case of \cite[Theorem~2.1]{sarma} for a commutative ring with identity.

\begin{corollary}\label{c5}
	For any ring $R$ of order $n$, we have
	\begin{align*}
		\text{(1)}&~zp_{_3}(R)\le \frac{3n^2 -3n|Z(R)|+|Z(R)|^3 +3(n-|Z(R)|)\displaystyle\sum_{x(\neq 0)\in Z(R)}|Ann(x)|}{n^3} \text{ and }\\ \medspace\\
		\text{(2)}&~zp_{_3}(R)\ge \frac{\splitfrac{3n^2-3n|Z(R)|+|Z(R)|^2+\displaystyle\sum_{x(\neq 0)\in Z(R)}\big\{(3n-|Z(R)|)|Ann(x)|}{-|Ann(x)|^2\big\}}}{n^3}.
	\end{align*}
\end{corollary}

\begin{theorem}\label{t2}
	Let $R$ be a ring of order $n$. Then we have
	\begin{enumerate}
		\item $\displaystyle zp_{{_k}}(R)\ge \frac{n^{k-1}+(n-|Z(R)|) |Ann_{{_{k-1}}}(R)|}{n^k}$, and the equality holds if and only if  $R$ is a field.
		
		\vspace{.3cm}
		
		\item $\displaystyle zp_{_k}(R)\le\frac{~\splitfrac{n^{k-1}+(n-|Z(R)|) |Ann_{{_{k-1}}}(R)|+(|Z(R)|-1)\big\{n^{k-1}}{-(n-|Z(R)|)^{k-1}\big\}}}{n^k},$ and the equality holds if and only if $Z(R)^2=0$.
	\end{enumerate}
\end{theorem}
\begin{proof}
	(1) The inequality follows from Theorem \ref{t1}(1). Again, it is clear that the equality occurs if and only if $Z(R)=\{0\}$, in other words, if and only if $R$ is a field.
	
	\medspace
	
	(2) We define $A_k, ~B_k, ~C_k$ and $C_k(x_1)$ for $x_1(\ne 0)\in Z(R)$ as in the proof of Theorem \ref{t1}. It is clear that $C_k(x_1)\subseteq\big\{(x_2,x_3,...,x_k)\in R^{k-1}:x_2x_3...x_k\in Z(R)\big\}$, and the equality holds if and only if $Ann(x_1)=Z(R)$. Again, since for any $x,y\in R$, we have $xy\in Z(R)$ if and only if $x\in Z(R)$ or $y\in Z(R)$, therefore, 
	\begin{align*}
		|C_k(x_1)|&\le \big|\big\{(x_2,x_3,...,x_k)\in R^{k-1}:x_2x_3...x_k\in Z(R)\big\}\big|\\
		&=|Z(R)|n^{k-2}+(n-|Z(R)|)|Z(R)|n^{k-3}+...+(n-|Z(R)|)^{k-2}|Z(R)|\\
		&=n^{k-1}-(n-|Z(R)|)^{k-1}.
	\end{align*}
	Thus, $|C_k|\le (|Z(R)|-1)\big\{n^{k-1}-(n-|Z(R)|)^{k-1}\big\}$, and the equality holds if and only if $Ann(x)=Z(R)$ for all $x(\ne 0)\in Z(R)$, equivalently, $Z(R)^2=0$. Hence, the equality in (2) holds if and only if $Z(R)^2=0$.
\end{proof}

\begin{remark}
	From Theorem \ref{t2}(2), we have $zp_{_2}(R)\le \frac{2(n-|Z(R)|)+|Z(R)|^2}{n^2}$ with equality if and only if  $Z(R)^2=0$, which is a case of  \cite[{Theorem~4.1}]{dolzan} for a commutative ring with identity.
\end{remark}

\begin{corollary} \label{c7}
	Let $R$ be a ring of order $n$. Then $R$ is a field if and only if
	\[zp_{_k}(R)=\frac{n^k-(n-1)^k}{n^k}.\]
\end{corollary}

In the following corollary, we establish a comparison between $zp_{_k}\hspace{-.05cm}(R)$ and $zp_{_{k-1}}\hspace{-.05cm}(R)$ for any ring $R$ and integer $k\ge3$.

\begin{corollary}\label{p6}
	For any ring $R$ of order $n$ and integer $k\ge 3$, we have
	\begin{align*}
		(1)& ~zp_{_k}(R)\ge \frac{n-|Z(R)|}{n}~zp_{_{k-1}}(R) + \frac{1}{n},\text{ and the equality holds if and only if $R$ is}\\ &\text{ a field}.\\
		(2)& ~zp_{_k}(R)\le \frac{n-|Z(R)|}{n}\hspace{.05cm}zp_{_{k-1}}\hspace{-.04cm}(R) + \frac{n^{k-1}\hspace{-.07cm}+\hspace{-.07cm}(|Z(R)|-1)\hspace{-.08cm}\left\{n^{k-1}-\hspace{-.07cm}(n-|Z(R)|)^{k-1}\right\}}{n^k},\\ &\text{ and the equality holds if and only if } Z(R)^2=0.
	\end{align*}
\end{corollary}

The bounds derived for $zp_{_k}(R)$ thus far are recursive. In the following, we derive an explicit form of upper bound for $zp_{_k}(R)$. Before we discuss our next result, we define the following:

\begin{definition}
	For any integer $d\ge 0$, we define the polynomial $P_{d}(x,y)$ of two real variables $x$ and $y$ of degree $d$ as
	\[P_{d}(x,y)=\sum\limits_{i=0}^{d}(-1)^{i}(i+1)\binom{d+2}{i+2} ~x^{d-i}~y^{i}.\] 
\end{definition}

\begin{lemma} \label{l2}
	For any integer $d\ge 0$ and real $x$, $P_{d}(x,x)=x^{d}$. 
\end{lemma}
\begin{proof}
	It is sufficient to prove that $$\sum\limits_{i=0}^{d}(-1)^{i}(i+1)\binom{d+2}{i+2}= 1, ~d\ge0.$$  Since for all integers $m\ge 0$ and real $x$, we have 
	\begin{equation} \label{e5} \sum\limits_{i=0}^{m}(-1)^{i}\binom{m}{i}(x-i)^{m-j}=0,~ 1\le j\le m~~\text{(see \cite{ruiz}}),\end{equation}
	therefore, putting $x=0$ and $j=m-1$ at (\ref{e5}), we get
	\begin{align*}
		&\sum\limits_{i=1}^{m}(-1)^{i}i\binom{m}{i}=0, ~m\ge1 \\
		\implies & \sum\limits_{i=0}^{m}(-1)^{i}\binom{m}{i} + \sum\limits_{i=0}^{m-2}(-1)^{i}(i+1)\binom{m}{i+2}=1, ~m\ge2.
	\end{align*}
	Again, putting $x=0$ and $j=m$ at (\ref{e5}), we have $\sum\limits_{i=0}^{m}(-1)^{i}\binom{m}{i}=0$. Thus,  \[\sum\limits_{i=0}^{d}(-1)^{i}(i+1)\binom{d+2}{i+2}= 1, ~d\ge0,\] where $d=m-2$.
\end{proof}

\begin{theorem} \label{t4}
	For any ring $R$ of order $n$ and integer $k\ge 2$, we have 
	\[zp_{_k}(R)\le\frac{n^k-\{n+(k-1)|Z(R)|-k\}(n-|Z(R)|)^{k-1}}{n^k}.\]
	The equality holds if and only if $Z(R)^2=0$. 
\end{theorem}
\begin{proof}
	From Theorem \ref{t2}(2), we have 
	\begin{multline}\label{e6}
		|Ann_{_k}(R)|\le n^{k-1}+(n-|Z(R)|) 
		|Ann_{{_{k-1}}}(R)|+(|Z(R)|-1)\\\big\{n^{k-1}-(n-|Z(R)|)^{k-1}\big\}.
	\end{multline}
	{\bf Claim:} $|Ann_{k}(R)|\le k(n-|Z(R)|)^{k-1}+|Z(R)|^2P_{k-2}(n,|Z(R)|)$.
	\medskip
	
	\noindent We prove this by induction on $k$.
	For $k=2$,
	\begin{align*}
		|Ann_{_2}(R)|&\le2(n-|Z(R)|)+|Z(R)|^2, \hspace{.1cm}\text{using (\ref{e6})}\\
		&=2(n-|Z(R)|)+|Z(R)|^2P_0(n,|Z(R)|).
	\end{align*}		
	Suppose that, for $k=l>2$ the result holds. That is,
	\[|Ann_{_l}(R)|\le l(n-|Z(R)|)^{l-1}+|Z(R)|^2P_{l-2}(n,|Z(R)|). \]
	Now,
	\begin{align*}
		&|Ann_{_{l+1}}(R)|\\
		&\le n^l+(n-|Z(R)|)|Ann_{_l}(R)|+(|Z(R)|-1)\big\{n^l-(n-|Z(R)|)^l\big\}\\
		&=n^l+(n-|Z(R)|)\big\{l(n-|Z(R)|)^{l-1}+|Z(R)|^2P_{l-2}(n,|Z(R)|)\big\}\\
		&\hspace{7.2cm}+(|Z(R)|-1)\big\{n^l-(n-|Z(R)|)^l\big\}	\\
		&=(l+1)(n-|Z(R)|)^l+|Z(R)|^2(n-|Z(R)|)\sum\limits_{i=0}^{l-2}(-1)^i(i+1)\binom{l}{i+2} n^{l-2-i} |Z(R)|^{i}\\& \hspace{7.1cm}+|Z(R)|\sum\limits_{i=1}^{l}(-1)^{i+1}\binom{l}{i}n^{l-i}|Z(R)|^{i}	\\
		&=(l+1)(n-|Z(R)|)^l+|Z(R)|^2\Bigg[\left\{\binom{l}{2}+\binom{l}{1}\right\}n^{l-1}+\sum\limits_{i=1}^{l-2}(-1)^i\bigg\{(i+1)\binom{l}{i+2}\\&\hspace{1.5cm}+i\binom{l}{i+1}+\binom{l}{i+1}\bigg\}n^{l-1-i}|Z(R)|^{i}+(-1)^{l-1}\big\{(l-1)+1\big\}|Z(R)|^{l-1}\Bigg]\\
		&=(l+1)(n-|Z(R)|)^l+|Z(R)|^2\bigg[\binom{l+1}{2}n^{l-1}+\sum\limits_{i=1}^{l-2}(-1)^i(i+1)\\
		&\hspace{5.85cm}\binom{l+1}{i+2}n^{l-1-i}|Z(R)|^{i}+(-1)^{l-1}l|Z(R)|^{l-1}\bigg]\\
		&=(l+1)(n-|Z(R)|)^l+|Z(R)|^2P_{l-1}(n,|Z(R)|).
	\end{align*}
	Therefore, the result is true for any $k\ge 2$.
	
	Next, we consider the function: $f_{_k}(x)=k(n-x)^{k-1}+x^2 P_{k-2}(n,x), ~x\in \mathbb{R}$. Then,
	\[f'_{_k}(x)=k(k-1)(x-1)(n-x)^{k-2}.\]
	Now, integrating we have
	\begin{align*}
		&\int f'_{_k}(x)dx + C = f_{_k}(x), \text{ where $C$ is an arbitrary constant}\\
		\implies & k(k-1)\int x(n-x)^{k-2}dx -k(k-1)\int (n-x)^{k-2}dx + C = k(n-x)^{k-1}\\&\hspace{9.2cm}+x^2 P_{k-2}(n,x)\\
		\implies & -\{n+(k-1)x\}(n-x)^{k-1}+C = x^2 P_{k-2}(n,x).
	\end{align*}
	At $x=n$, $C=n^2P_{k-2}(n,n)=n^{k}$ (using Lemma \ref{l2}). Thus, 
	\[x^2 P_{k-2}(n,x)=n^k-\{n+(k-1)x\}(n-x)^{k-1}.\] This implies \[f_{_k}(x)=n^k-\{n+(k-1)x-k\}(n-x)^{k-1}.\] Hence,
	\[zp_{_k}(R)\le\frac{n^k-\{n+(k-1)|Z(R)|-k\}(n-|Z(R)|)^{k-1}}{n^k}.\]
\end{proof}

\section{$zp_{_k}(R)$ for local rings}

We know that any finite commutative ring with identity can be written uniquely (up to isomorphism) as a direct product of local rings (see \cite[Theorem 8.7]{atiyah}). In addition, $zp_{_k}$ is multiplicative with respect to direct product of finite rings, which we shall prove shortly. Therefore, investigating $zp_{_k}(R)$ for a local ring is an important tool to investigate the same for a commutative ring with identity.

\begin{lemma}\label{l1}
	For any two rings $R_1$ and $R_2$, we have  $$zp_{_k}(R_1\times R_2)=zp_{_k}(R_1) zp_{_k}(R_2).$$
\end{lemma}
\begin{proof}
	We have $zp_{_k}(R_1\times R_2)=|Ann_{_k}(R_1\times R_2)|/|R_1\times R_2|^k$. So, it is sufficient to prove that $|Ann_{_k}(R_1\times R_2)|=|Ann_{_k}(R_1)||Ann_{_k}(R_2)|$. Now,
	\begin{align*}
		|Ann_{_k}(R_1\times R_2)|&=\big|\big\{\big((r_1,r_1'),(r_2,r_2'),...,(r_k,r_k')\big)\in (R_1\times R_2)^k : (r_1,r_1')(r_2,r_2')...\\& \hspace{7.4cm}(r_k,r_k')=(0,0)\big\}\big|\\
		&=\big|\big\{\big((r_1,r_1'),(r_2,r_2'),...,(r_k,r_k')\big)\in (R_1\times R_2)^k : r_1r_2...r_k=0,\\ & \hspace{7.7cm}r_1'r_2'...r_k'=0\big\}\big|.
	\end{align*}
	Since $R_1$ and $R_2$ are commutative, therefore, $(R_1\times R_2)^k\cong R_1^k\times R_2^k$, and so,
	\begin{align*}
		|Ann_{_k}(R_1\times R_2)| &=\big|\big\{\big((r_1,r_2,...,r_k),(r_1',r_2',...,r_k')\big)\in R_1^k\times R_2^k : r_1r_2...r_k=0,\\&\hspace{7.7cm}r_1'r_2'...r_k'=0\big\}\big|\\
		&=\big|\big\{(r_1,r_2,...,r_k)\in R_1^k : r_1r_2...r_k=0\big\}\\&\hspace{4cm}\times\big\{(r_1',r_2',...,r_k')\in R_2^k : r_1'r_2'...r_k'=0\big\}\big|\\
		&=|Ann_{_k}(R_1)||Ann_{_k}(R_2)|.
	\end{align*}
\end{proof}

\begin{lemma}\label{l3}
	Let $R$ be a finite local ring. Then $Z(R)$ is the unique maximal ideal of $R$. Further, $|R|=p^{nr}$ and $|Z(R)|=p^{(n-1)r}$ for some prime $p$ and positive integers $n, r$.  
\end{lemma}
\begin{proof}
	See \cite[Theorem~2]{ragha}.
\end{proof}

\begin{proposition} \label{t5}
	Let $R$ be a local ring with $|R|=p^\alpha$ for some prime $p$ and positive integer $\alpha$. Then
	\[zp_{_k}(R)\le\frac{p^{k\alpha}-\left\{p^\alpha+(k-1)p^{i}-k\right\}(p^\alpha-p^i)^{k-1}}{p^{k\alpha}}\]
	for some $0\le i\le \alpha -1$, where $|Z(R)|=p^i$. The equality holds if and only if $Z(R)^2=0$.
\end{proposition}
\begin{proof}
	It is clear from Theorem \ref{t4}.
\end{proof}

\begin{lemma} \label{l4}
	Let $p$ be a prime and $\alpha$ be a positive integer. If $0\le i<j<\alpha$ for integers $i$ and $j$, then  \[p^{k\alpha}-\left\{p^\alpha+(k-1)p^{i}-k\right\}(p^\alpha-p^i)^{k-1}<p^{k\alpha}-\left\{p^\alpha+(k-1)p^{j}-k\right\}(p^\alpha-p^j)^{k-1}.\]
\end{lemma}
\begin{proof}
	It follows from the fact that the function $f_{_k}(x)=n^k-\{n+(k-1)x-k\}(n-x)^{k-1}$ in Theorem \ref{t4} is strictly increasing for real $x\in(1,n)$. Therefore, $f_{_k}(p^i)<f_{_k}(p^j)$ for $0\le i<j<\alpha$.
\end{proof}

\begin{definition}
	For a commutative ring $R$ with identity and a unitary $R$-module $M$, $R*M$ is called the {\it Nagata's idealization} of $M$ in $R$, which is the ring $R\times M$, where for any $(r_1,m_1),(r_2,m_2)\in R\times M$, $(r_1,m_1)+(r_2,m_2)=(r_1+r_2,m_1+m_2) \text{ and }(r_1,m_1)\cdot(r_2,m_2)=(r_1r_2,r_1m_2+r_2m_1).$
\end{definition}

Note that, (1,0) is the identity of $R*M$ and $0*M$ is an ideal of $R*M$ of nilpotency index 2. Also, $R*M$ is local if and only if $R$ is local.

\begin{lemma}[{\cite[Theorem 25.3]{huckaba}}]\label{l5}
	For a commutative ring $R$ with identity and unitary $R$-module $M$, $Z(R*M)=\left(Z(R)\cup Z(M)\right)\times M$, where $Z(M)=\{r\in R : rm = 0 \text{ for some }m(\ne0)\in M\}$.
\end{lemma}

\begin{example}\label{l6}
	For prime $p$ and positive integer $\alpha$, the ring $R=\mathbb{Z}_p*(\mathbb{Z}_p)^{\alpha -1}$ is local with $Z(R)=0*(\mathbb{Z}_p)^{\alpha -1}$. Therefore, by Proposition \ref{t5},
	\[zp_{_k}(R)=\frac{p^{\alpha-1}\left\{p^k-(k+p-1)(p-1)^{k-1}\right\}+k(p-1)^{k-1}}{p^{k+\alpha-1}}.\]
\end{example}

\begin{proposition}[{\cite[Proposition~2.3]{sarma}}] \label{l7}
	Let $(R,\mathcal{M})$ be a local ring with $|R|=p^\alpha$ and $|\mathcal{M}|=p^{\alpha -1}$ for some prime $p$ and positive integer $\alpha$. Then $\mathcal{M}^2=0$ if and only if $R\cong \mathbb{Z}_p*(\mathbb{Z}_p)^{\alpha -1}$ or $R\cong\mathbb{Z}_{p^2}*(\mathbb{Z}_{p})^{\alpha -2}.$
\end{proposition}

\begin{theorem} \label{p1}
	Let $(R, \mathcal{M})$ be a local ring with $|R|=p^\alpha$ for some prime $p$ and positive integer $\alpha$. Then 
	\[zp_{_k}(R)\le\frac{p^{\alpha-1}\left\{p^k-(k+p-1)(p-1)^{k-1}\right\}+k(p-1)^{k-1}}{p^{k+\alpha-1}}.\]
	The equality holds if and only if $R\cong \mathbb{Z}_p * (\mathbb{Z}_p)^{\alpha-1}$ or $R\cong \mathbb{Z}_{p^2} * (\mathbb{Z}_p)^{\alpha-2}$.
\end{theorem}
\begin{proof}
	The inequality follows from Proposition \ref{t5} and Lemma \ref{l4}. The equality holds if and only if $|\mathcal{M}|=p^{\alpha-1}$ with $\mathcal{M}^2=0$, or, in other words, if and only if $R\cong \mathbb{Z}_p * (\mathbb{Z}_p)^{\alpha-1}$ or $R\cong \mathbb{Z}_{p^2} * (\mathbb{Z}_p)^{\alpha-2}$.
\end{proof}

This leads to the next result.

\begin{theorem} \label{t6}
	Let $R$ be a ring such that $R\cong R_1\times R_2\times ...\times R_t$, where each $R_i$ is local ring with identity of order $p^{\alpha_i}_i$, $p_i'$s are distinct primes and $\alpha_i>0$ for $1\le i\le t$. Then
	\[zp_{_k}(R)\le \prod_{i=1}^{t}\frac{p_i^{\alpha_i-1}\left\{p_i^k-(k+p_i-1)(p_i-1)^{k-1}\right\}+k(p_i-1)^{k-1}}{p_i^{k+\alpha_i-1}},\] and the equality holds if and only if for each $i$, $R_i\cong \mathbb{Z}_{p_i}*(\mathbb{Z}_{p_i})^{\alpha_i -1}$ or $R\cong\mathbb{Z}_{{p_i}^2}*(\mathbb{Z}_{p_i})^{\alpha_i -2}.$ 
\end{theorem}
\begin{proof}
	The proof is clear from Lemma \ref{l1} and Theorem \ref{p1}.
\end{proof}

\begin{remark}
	Note that, the inequality in Theorem \ref{p1}, however, need not hold for all {\it $p$-rings}, rings whose order is power of the prime $p$. For example, $R=\mathbb{Z}_2\times \mathbb{Z}_2$.
\end{remark}

In the following, we establish the largest possible upper bound for $zp_{_k}(R)$ for any $p$-ring $R$.
\begin{proposition}\label{c6}
	For any $p$-ring $R$, where $p$ is a prime, we have
	\[zp_{_k}(R)\le\frac{p^k-(p-1)^k}{p^k}.\]
	The equality holds if and only if $R\cong \mathbb{Z}_p$.
\end{proposition}
\begin{proof}
	Suppose that $|R|=p^\alpha$ for some integer $\alpha>0$. First, assume that $R$ is local. Then 
	\begin{align*}
		zp_{_k}(R)&\le\frac{p^{\alpha-1}\left\{p^k-(k+p-1)(p-1)^{k-1}\right\}+k(p-1)^{k-1}}{p^{k+\alpha-1}}\\
		&\le\frac{p^k-(k+p-1)(p-1)^{k-1}}{p^{k}}+\frac{k(p-1)^{k-1}}{p^{k+\alpha-1}}.
	\end{align*}
	As the value of the term $\frac{k(p-1)^{k-1}}{p^{k+\alpha-1}}$ decreases as $\alpha$ increases, therefore, we have $zp_{_k}(R)\le\frac{p^k-(p-1)^k}{p^k}$ with equality if and only if $\alpha=1$.
	
	\medspace
	
	Next, assume that $R$ is non-local. Then, $R$ is a direct product of local rings whose order is divisible by $p$. Thus,  by Lemma \ref{l1}, $zp_{_k}(R)<\frac{p^k-(p-1)^k}{p^k}$.
\end{proof}

\begin{corollary}\label{c2}
	Let $R$ be a ring with $|R|=\prod\limits_{i=1}^{t}p_i^{\alpha_i}$, where $p_{i}'$s are distinct primes and $\alpha_i>0$ for $1\le i\le t$. Then
	\[\prod_{i=1}^{t}\frac{p_{i}^{k\alpha_i}-(p_{i}^{\alpha_i}-1)^k}{p^{k\alpha_i}_i}\le zp_{_k}(R)\le \prod\limits_{i=1}^{t}\frac{p_i^k-(p_i-1)^k}{p_{i}^{k}}.\]
	Also,
	\begin{enumerate}
		\item $ zp_{_k}(R)=\prod\limits_{i=1}^{t}\frac{p_{i}^{k\alpha_i}-(p_{i}^{\alpha_i}-1)^k}{p^{k\alpha_i}_i}$ if and only if $R\cong GF(p_1^{\alpha_1})\times GF(p_2^{\alpha_2}) \times ...$ $\times GF(p_t^{\alpha_t})$, and
		
		\item $zp_{_k}(R)= \prod\limits_{i=1}^{t}\frac{p_i^k-(p_i-1)^k}{p_{i}^{k}}$ if and only if $R\cong \mathbb{Z}_{p_1}\times \mathbb{Z}_{p_2} \times ...\times \mathbb{Z}_{p_t}$.
	\end{enumerate}
\end{corollary}
\begin{proof}
	Since $R\cong R_1\times R_2 \times ...\times R_t$, where each $R_i$ is a ring with identity of order $p_i^{\alpha_i}$ for $1\le i\le t$, therefore, by Corollary \ref{c7} and Proposition \ref{c6}, for each $i$, $1\le i\le t$, we have
	\[\frac{p_{i}^{k\alpha_i}-(p_{i}^{\alpha_i}-1)^k}{p^{k\alpha_i}_i}\le zp_{_k}(R_i)\le\frac{p_i^k-(p_i-1)^k}{p_{i}^{k}}.\]
	Moreover,
	\[zp_{_k}(R_i)=\frac{p_{i}^{k\alpha_i}-(p_{i}^{\alpha_i}-1)^k}{p^{k\alpha_i}_i} \iff R_i\cong GF(p_i^{\alpha_i})\] and
	\[zp_{_k}(R_i)=\frac{p_i^k-(p_i-1)^k}{p_{i}^{k}} \iff R_i\cong \mathbb{Z}_{p_i}.\]
	The remaining part follows accordingly.
\end{proof}

\begin{proposition}\label{p2}
	Let $R$ be a ring of order $p^2$, where $p$ is a prime. Then
	\begin{enumerate}
		\item $R$ is local if and only if $ zp_{_k}(R)$ is either $$\dfrac{p^{2k}-(p^2-1)^k}{p^{2k}} \text{ or }  \dfrac{p\big\{p^{k}-(k+p-1)(p-1)^{k-1}\big\}+k(p-1)^{k-1}}{p^{k+1}}.$$
		
		\item $R$ is non-local if and only if $zp_{_k}(R)$ is $$\left(\dfrac{p^{k}-(p-1)^k}{p^{k}}\right)^2.$$
	\end{enumerate}
\end{proposition}
\begin{proof}
	(1) $R$ is local if and only if $R$ is isomorphic to one of the rings: $GF(p^2),~ \mathbb{Z}_{p^2}$ or $\mathbb{Z}_p * \mathbb{Z}_p$. By Corollary \ref{c7}, $$R\cong GF(p^2) \iff zp_{_k}(R)=\frac{p^{2k}-(p^2-1)^k}{p^{2k}},$$ and by Theorem \ref{p1}, $$R\cong \mathbb{Z}_{p^2} \text{ or } \mathbb{Z}_p * \mathbb{Z}_p \iff zp_{_k}(R)=\frac{p\left\{p^{k}-(k+p-1)(p-1)^{k-1}\right\}+k(p-1)^{k-1}}{p^{k+1}}.$$
	
	\medspace
	
	(2) Suppose that $R$ is non-local. Then $R\cong  R_1\times R_2$, where $R_1$ and $R_2$ are local rings with identity. This implies $|R_1|=|R_2|=p$, and so, $R\cong \mathbb{Z}_p\times \mathbb{Z}_p$ with $zp_{_k}(R)=\left(\frac{p^k-(p-1)^k}{p^k} \right)^2$. The converse part follows from (1).
\end{proof}

\begin{example}\label{p3}
	Let $n$ be an integer such that $n=\prod\limits_{i=1}^{t}p^{\alpha_i}_i$, where $p_{i}'$s are distinct primes and $\alpha_i>0$ for $1\le i\le t$. Then
	\[zp_{_k}(\mathbb{Z}_n)\le \prod_{i=1}^{t}\frac{p_i^{\alpha_i-1}\left\{p_i^k-(k+p_i-1)(p_i-1)^{k-1}\right\}+k(p_i-1)^{k-1}}{p_i^{k+\alpha_i-1}}.\]
	In particular, if $\alpha_i\in\{1,2\}$ for $1\le i\le t$. Then $zp_{_k}(\mathbb{Z}_n)=\prod\limits_{i=1}^t zp_{_k}(\mathbb{Z}_{p_i^{\alpha_i}}),$ where
	$$zp_{_k}(\mathbb{Z}_{p_i^{\alpha_i}}) = \begin{cases}
		\frac{p_i^k-(p_i-1)^k}{p_i^k},  & \text{ if }\alpha_i=1  \\
		\frac{p_i\left\{p_i^{k}-(k+p_i-1)(p_i-1)^{k-1}\right\}+k(p_i-1)^{k-1}}{p_i^{k+1}}, & \text{ if }\alpha_i=2.
	\end{cases}$$
\end{example}

\begin{definition}
	For any prime $p$ and integers $\alpha>0, k\ge2$, we define
	\[\mathcal{B}_k(p; \alpha)=\frac{p^{\alpha-1}\left\{p^k-(k+p-1)(p-1)^{k-1}\right\}+k(p-1)^{k-1}}{p^{k+\alpha-1}}.\]
\end{definition}

\begin{remark}\label{r1}
	Note that, if $R$ is a local ring of order $p^\alpha$ for some prime $p$ and positive integer $\alpha$, then  $zp_{_k}(R)\le\mathcal{B}_k(p; \alpha)$, for $k\ge2$; and the equality holds if and only if $R\cong \mathbb{Z}_p * (\mathbb{Z}_p)^{\alpha-1}$ or $R\cong \mathbb{Z}_{p^2} * (\mathbb{Z}_p)^{\alpha-2}$.
\end{remark}

\begin{lemma}\label{p4}
	The following statements hold:
	\begin{enumerate}
		\item $\mathcal{B}_k(p_1; \alpha)>\mathcal{B}_k(p_2; \alpha)$ for primes $p_1<p_2$ and integers $0<\alpha, 2\le k.$
		
		\item $\mathcal{B}_k(p; \alpha_1)>\mathcal{B}_k(p; \alpha_2)$ for prime $p$ and integers $0<\alpha_1<\alpha_2, 2\le k.$
		
		\item $\mathcal{B}_k(2; 1)\ge\mathcal{B}_k(p; \alpha)$ for any prime $p$ and integers $0<\alpha, 2\le k.$ The equality holds if and only if $p=2$ and $\alpha=1$.
		
		\item $\mathcal{B}_k(3; 1)>\mathcal{B}_k(2; 2)$ for $k=2,3$ and $\mathcal{B}_k(3; 1)<\mathcal{B}_k(2; 2)$ for $k\ge4$.
	\end{enumerate}
\end{lemma}
\begin{proof}
	(1) For real $x>0$ and fixed $k$ and $\alpha$, we have
	\[
	\frac{d}{dx}\mathcal{B}_k(x; \alpha)=\frac{-k(x-1)^{k-2}}{x^{k+\alpha}}\Big\{(k-1)(x^{\alpha-1}-1)+\alpha(x-1)\Big\}.
	\]
	Since for $x>1$, $\frac{d}{dx}\mathcal{B}_k(x; \alpha)<0$, therefore,  for primes $p_1<p_2$, $\mathcal{B}_k(p_1; \alpha)>\mathcal{B}_k(p_2; \alpha)$. 
	
	\medskip
	
	(2) Since
	\[
	\mathcal{B}_k(p; \alpha) = \frac{p^k-(k+p-1)(p-1)^{k-1}}{p^{k}}+\frac{k(p-1)^{k-1}}{p^{k+\alpha-1}},
	\]
	and the term $\frac{k(p-1)^{k-1}}{p^{k+\alpha-1}}$ decreases as $\alpha$ increases, hence the required result follows.
	
	\medspace
	
	(3) It follows from (1) and (2).
	
	\medspace
	
	(4) It is clear that for $k=2 ~\text{and}~3$, $\mathcal{B}_k(3; 1)>\mathcal{B}_k(2; 2)$. Suppose that for $k\ge 4$, $\mathcal{B}_k(3; 1)\ge\mathcal{B}_k(2; 2)$. Then  $\frac{3^k-2^k}{3^k}\ge \frac{2^{k+1}-k-2}{2^{k+1}}$, which implies $2^{2k+1}\le3^k(k+2)$, which is a contradiction. Hence, for $k\ge4$, $\mathcal{B}_k(3; 1)<\mathcal{B}_k(2; 2)$.
\end{proof}

\begin{proposition} \label{p5}
	Let $(R_1,\mathcal{M}_1)$ and $(R_2,\mathcal{M}_2)$ be two local rings.
	\begin{enumerate}
		\item If $|R_1|=p_1^\alpha, |\mathcal{M}_1|=p_1^{\alpha-1}$ with $\mathcal{M}_1^2=0$, and $|R_2|=p_2^\alpha, |\mathcal{M}_2|=p_2^{\alpha-1}$ with $\mathcal{M}_2^2=0$ for primes $p_1<p_2$ and fixed integer $\alpha>0$, then $zp_{_k}(R_1)>zp_{_k}(R_2)$.
		
		\item If $|R_1|=p^{\alpha_1}, |\mathcal{M}_1|=p^{\alpha_1-1}$ with $\mathcal{M}_1^2=0$, and $|R_2|=p^{\alpha_2}, |\mathcal{M}_2|=p^{\alpha_2-1}$ with $\mathcal{M}_2^2=0$ for prime $p$ and integers $0<\alpha_1<\alpha_2$, then $zp_{_k}(R_1)>zp_{_k}(R_2)$.
	\end{enumerate}
\end{proposition}
\begin{proof}
	It follows from Lemma \ref{p4}.
\end{proof}

\begin{remark}\label{r2}
	Note that, Proposition \ref{p5}(2) is a generalization of \cite[Proposition~2.9]{emskhani2}.
\end{remark}

\section{Classification of some rings by $zp_{_k}(R)$}
In this section, we show that $zp_{_k}(R)\le\mathcal{B}_k(2; 1)$ with equality if and only if $R\cong\mathbb{Z}_2$. Next, we classify all rings $R$ with $zp_{_k}(R)\ge\mathcal{B}_k(2; 3), ~ \forall~ k\ge2$. Finally, we find all possible primes $p$ that divide the order of local rings  satisfying certain condition.

\begin{theorem}\label{t9}
	For any ring $R$ and integer $k\ge2$, we have $$zp_{_k}(R)\le\mathcal{B}_k(2; 1).$$ The equality holds if and only if $R\cong \mathbb{Z}_2$.
\end{theorem}
\begin{proof}
	First suppose that $R$ is local. Then $|R|=p^\alpha$ for some prime $p$ and positive integer $\alpha$. So, $zp_{_k}(R)\le \mathcal{B}_k(p; \alpha)$, and by Lemma \ref{p4}, we have $zp_{_k}(R)\le \mathcal{B}_k(p; \alpha)\le \mathcal{B}_k(2; 1)$. Again, the equality holds if and only if $R\cong\mathbb{Z}_2$.
	
	Next, suppose that $R$ is non-local. Then $R\cong R_1\times...\times R_t, t\ge 2$, where each $R_i$ is local with $zp_{_k}(R_i)\le\mathcal{B}_k(2; 1)$, and hence, $zp_{_k}(R)<\mathcal{B}_k(2; 1)$.
\end{proof}

\begin{theorem}\label{t10}
	Let $R$ be a local ring. Then
	\begin{enumerate}
		\item for $k=2$ and $3$, $zp_{_k}(R)\notin(\mathcal{B}_k(3; 1), \mathcal{B}_k(2; 1))$ and for $k\ge4$, $zp_{_k}(R)\notin(\mathcal{B}_k(2; 2), \mathcal{B}_k(2; 1))$. Moreover, $zp_{_k}(R)=\mathcal{B}_k(3; 1)$ if and only if $R\cong\mathbb{Z}_3$, and $zp_{_k}(R)=\mathcal{B}_k(2; 2)$ if and only if $R\cong\mathbb{Z}_4$ or $R\cong\mathbb{Z}_2*\mathbb{Z}_2$.
		
		\item if $zp_{_k}(R)\ge\mathcal{B}_k(2; 3)$, then $R$ is isomorphic to one of the rings:
		\begin{enumerate}
			\item[{(i)}] $\mathbb{Z}_2, \mathbb{Z}_3, \mathbb{Z}_4, \mathbb{Z}_2 * \mathbb{Z}_2, \mathbb{Z}_4 * \mathbb{Z}_2, \mathbb{Z}_2 * (\mathbb{Z}_2)^2, GF(4)$ when $k=2$.
			
			\medspace
			
			\item[{(ii)}] $\mathbb{Z}_2, \mathbb{Z}_3, \mathbb{Z}_4, \mathbb{Z}_2 * \mathbb{Z}_2, \mathbb{Z}_4 * \mathbb{Z}_2, \mathbb{Z}_2 * (\mathbb{Z}_2)^2$ when $2<k<6$.
			
			\medspace
			
			\item[{(iii)}] $\mathbb{Z}_2, \mathbb{Z}_4, \mathbb{Z}_2 * \mathbb{Z}_2, \mathbb{Z}_4 * \mathbb{Z}_2, \mathbb{Z}_2 * (\mathbb{Z}_2)^2$ when $6\le k$.
		\end{enumerate}
	\end{enumerate}
\end{theorem}
\begin{proof}
	(1) Since $R$ is local, therefore, $|R|=p^\alpha$ for some prime $p$ and positive integer $\alpha$. First, suppose that $k=2$ or 3. Assume that $zp_{_k}(R)>\mathcal{B}_k(3; 1)$. This implies $p=2$, and therefore, $|R|=2^\alpha$. Now, for $\alpha\ge 2$, $zp_{_k}(R)\le \mathcal{B}_k(2; 2)$, and since for $k=2$ and 3, $\mathcal{B}_k(2; 2)<\mathcal{B}_k(3; 1)$, this implies $\alpha=1$. In other words, $R\cong \mathbb{Z}_2$, and therefore,  $zp_{_k}(R)\notin(\mathcal{B}_k(3; 1), \mathcal{B}_k(2; 1))$, and by Remark \ref{r1}, $zp_{_k}(R)= \mathcal{B}_k(3; 1)$ if and only if $R\cong\mathbb{Z}_3$. Next, suppose that $k\ge4$. Assume that $zp_{_k}(R)>\mathcal{B}_k(2; 2)$. For $p\ge3$, we have $zp_{_k}(R)\le\mathcal{B}_k(3; 1)$. But for $k\ge4$, $\mathcal{B}_k(3; 1)<\mathcal{B}_k(2; 2)$, a contradiction. So, $p=2$, and therefore, $|R|=2^\alpha$. Also, for $\alpha\ge2$, $zp_{_k}(R)\le\mathcal{B}_k(2; 2)$. Thus, $\alpha=1$, which implies $R\cong\mathbb{Z}_2$. Therefore,  $zp_{_k}(R)\notin(\mathcal{B}_k(2; 2), \mathcal{B}_k(2; 1))$, and  $zp_{_k}(R)= \mathcal{B}_k(2; 2)$ if and only if $R\cong\mathbb{Z}_4$ or $R\cong\mathbb{Z}_2 * \mathbb{Z}_2$.
	
	\medskip
	
	(2) Suppose that $zp_{_k}(R)\ge\mathcal{B}_k(2; 3)$ and $|R|=p^\alpha$ for some prime $p$ and positive integer $\alpha$. 
	
	(i) For $k=2$ it is already proved, see \cite{emskhani1}. 
	
	(ii) Let $2<k<6$. For $k=3$ it is proved in \cite{sarma}. Observe that $zp_{_4}(R)\ge\mathcal{B}_4(2; 3)$ if and only if $zp_{_5}(R)\ge\mathcal{B}_5(2; 3)$. Hence, to complete the proof of (ii), it is enough to classify rings with $zp_{_4}(R)\ge\mathcal{B}_4(2; 3)$.  So, let $k=4$. Then
	\[\frac{3}{4}\le zp_{_4}(R)\le \frac{p^{\alpha-1}\left\{p^4-(p+3)(p-1)^{3}\right\}+4(p-1)^{3}}{p^{\alpha+3}}\le \frac{p^4-(p-1)^4}{p^4}.\]
	This implies $p\le3$. Let $p=2$. Then
	\[\frac{3}{4}\le\frac{11\cdot2^{\alpha+1}+4}{2^{\alpha+3}}.\]
	This implies  $\alpha\le 3$. For $\alpha=1$, $|R|=2$, that is, $R\cong\mathbb{Z}_2$ with $zp_{_4}(R)=\frac{15}{16}$. For $\alpha=2$, $|R|=4$ with $|\mathcal{M}|=1$ or 2. If $|\mathcal{M}|=1$, then $R\cong GF(4)$ with $zp_{_4}(R)=\frac{175}{256}<\frac{3}{4}$, a contradiction. If $|\mathcal{M}|=2$, then $R\cong\mathbb{Z}_4$ or $R\cong\mathbb{Z}_2*\mathbb{Z}_2$ with $zp_{_4}(R)=\frac{13}{16}$. For $\alpha=3$, $|R|=8$ with $|\mathcal{M}|=1$ or 4. If $|\mathcal{M}|=1$, then $R\cong GF(8)$ with $zp_{_4}(R)=\frac{1695}{4096}<\frac{3}{4}$, a contradiction. If $|\mathcal{M}|=4$, then there are two possibilities: $\mathcal{M}^2\ne0$ and $\mathcal{M}^2=0$. Suppose that  $\mathcal{M}^2\ne0$. Then by Remark \ref{r1}, $zp_{_4}(R)<\mathcal{B}_4(2; 3)$, a contradiction. Thus, $\mathcal{M}^2=0$, which implies $R\cong\mathbb{Z}_2*(\mathbb{Z}_2)^2$ or $R\cong\mathbb{Z}_4*\mathbb{Z}_2$ with $zp_{_4}(R)=\frac{3}{4}$. Next, let $p=3$. Then 
	\[\frac{3}{4}\le zp_{_4}(R)\le\frac{33\cdot3^{\alpha-1}+32}{3^{\alpha+3}},\] 
	which implies $\alpha= 1$. Thus, $|R|=3$, or $R\cong\mathbb{Z}_3$ with $zp_{_4}(R)=\frac{65}{81}$. Therefore, the possible local rings $R$ with $zp_{_4}(R)\ge\mathcal{B}_4(2; 3)$ are: $\mathbb{Z}_2, \mathbb{Z}_3, \mathbb{Z}_4, \mathbb{Z}_2 * \mathbb{Z}_2, \mathbb{Z}_4 * \mathbb{Z}_2, \mathbb{Z}_2 * (\mathbb{Z}_2)^2$.
	
	(iii) Assume that $k\ge6$. Then, 
	\[\mathcal{B}_k(2; 3)=\frac{2^{k+2}-3k-4}{2^{k+2}}\le zp_{_k}(R)\le\frac{p^k-(p-1)^k}{p^k}\]
	\[\implies\frac{3k+4}{2^{k+2}}\ge\left(\frac{p-1}{p}\right)^k.\] Assume that $p>2$, that is, $p=2+l$ for some integer $l\ge 1$. Thus,
	\[\frac{3k+4}{2^{k+2}}\ge\left(\frac{l+1}{l+2}\right)^k\ge\left(\frac{2}{3}\right)^k\]\[\implies k\cdot3^{k+1}+4\cdot3^k\ge2^{2k+2}, \text{ a contradiction.}\] Therefore, $p=2$. Thus, $\mathcal{B}_k(2; 3)\le zp_{_k}(R)\le \mathcal{B}_k(2; \alpha)$ implies $\alpha\le3$. For $\alpha=1$, $|R|=2$, or $R\cong\mathbb{Z}_2$ with $zp_{_k}(R)=\frac{2^k-1}{2^k}$. For $\alpha=2$, $|R|=4$ with $|\mathcal{M}|=1$ or 2. If $|\mathcal{M}|=1$, then $R\cong GF(4)$ with $zp_{_k}(R)=\frac{4^k-3^k}{4^k}<\mathcal{B}_k(2; 3)$, a contradiction. If $|\mathcal{M}|=2$, then $R\cong \mathbb{Z}_4$ or $R\cong \mathbb{Z}_2*\mathbb{Z}_2$ with $zp_{_k}(R)=\frac{2^{k+1}-k-2}{2^{k+1}}$. Now, for $\alpha=3$, $|R|=8$ with $|\mathcal{M}|=1$ or 4. If $|\mathcal{M}|=1$, then $R\cong GF(8)$ with $zp_{_k}(R)=\frac{8^k-7^k}{8^k}<\mathcal{B}_k(2; 3)$, a contradiction. Again, if $|\mathcal{M}|=4$, then as proved earlier $\mathcal{M}^2=0$, which implies $R\cong \mathbb{Z}_4 * \mathbb{Z}_2$, or $R\cong \mathbb{Z}_2 * (\mathbb{Z}_2)^2$ with $zp_{_k}(R)=\frac{2^{k+2}-3k-4}{2^{k+2}}$. This completes the proof.
\end{proof}

\begin{corollary}\label{c4}
	There are exactly 7, 6 and 5 local rings (up to isomorphism) with $zp_{_k}(R)\ge \mathcal{B}_k(2; 3)$ for $k=2$, $2< k<6$ and $6\le k$, respectively.
\end{corollary}

Since for a local ring $R$ of order $p^\alpha$ for some prime $p$ and positive integer $\alpha$, $zp_{_k}(R)\le\mathcal{B}_k(p;\alpha),$ and so, $zp_{_k}(R)\le\mathcal{B}_k(2;\alpha)$, therefore, we get 
$$zp_{_k}(R)\le \frac{2^k-k-1}{2^k}+\frac{k}{2^{k+\alpha-1}}.$$
We conclude this paper discussing our next result in which we obtain all possible primes $p$ which divide the order of the local rings $R$ that satisfy the condition: $zp_{_k}(R)\ge\frac{2^k-k-1}{2^k}$ for $k\ge 2$.

\begin{theorem}\label{t8}
	Let $R$ be a local ring such that  $zp_{_k}(R)\ge \frac{2^{k}-k-1}{2^{k}}$, where $k \ge 2$ is an integer. If $p$ is a prime such that $p~\big | |R|$, then
	\begin{enumerate}
		\item $k=2$, only if $p\in\{2,3,5,7\}$.
		\item $2<k<8$, only if  $p\in\{2,3\}$.
		\item $8\le k$, only if $p=2$.
	\end{enumerate}
\end{theorem}
\begin{proof}
	Let $|R|=p^\alpha$ for some positive integer $\alpha$. Then
	\begin{align*}
		&\frac{2^{k}-k-1}{2^{k}}\le zp_{_k}(R)\le \mathcal{B}_k(p;\alpha)\le \mathcal{B}_k(p;1) \\
		\implies&\frac{2^{k}-k-1}{2^{k}}\le\frac{p^k-(p-1)^{k}}{p^k}\\
		\implies & \left(\frac{p-1}{p}\right)^k\le \frac{k+1}{2^k}.
	\end{align*}
	
	(1) Suppose that $k=2$. Then $\left(\frac{p-1}{p}\right)^2\le\frac{3}{4}$, which implies $p\le7$. 
	
	(2) Suppose that $k=3$. Then  $\left(\frac{p-1}{p}\right)^3\le\frac{1}{2}$, which implies $p\le3$. Similarly, it can be shown that for $k=4,5,6,7$, the possible values of $p$ are 2 and 3.
	
	(3) Now, suppose that $k\ge8$. We assume that $p\ge 3$. Then $\left(\frac{p-1}{p}\right)^k\ge \left(\frac{2}{3}\right)^k$. This implies $\left(\frac{2}{3}\right)^k\le\frac{k+1}{2^k}$, or $\left(\frac{4}{3}\right)^k\le k+1$, which is a contradiction. Thus, the only possible value of $p$ is 2.
\end{proof}

\bibliographystyle{amsplain}

\begin{thebibliography}{15}
	
	\bibitem{atiyah}M. F. Atiyah and I. G. MacDonald, \textit{Introduction to commutative algebra}, Addison-Wesley, Reading, MA, 1969.
	
	\bibitem{buckley1}S. M. Buckley and D. MacHale, \textit{Commuting probability for subrings and quotient rings}, J. Algebra Comb. Discrete Struct. Appl. \textbf{4}(2) (2016), 189-196.
	
	\bibitem{buckley2}S. M. Buckley and D. MacHale, \textit{Contrasting the commuting probabilities of groups and rings}, Preprint, https://archive.maths.nuim.ie/staff/sbuckley/Papers/bm\_g-vs-r.pdf.
	
	\bibitem{buckley3}S. M. Buckley, D. MacHale and \'A. N\'i Sh\'e, \textit{Finite rings with many commuting pairs of elements}, Preprint, 2014, https://archive.maths.nuim.ie/staff/sbuckley/Papers/bms.pdf.
	
	\bibitem{dolzan}D. Dol\v zan, \textit{The probability of zero multiplication in finite rings}, Bull. Aus. Math. Soc. \textbf{106}(1) (2022), 83-88.
	
	\bibitem{dutta}J. Dutta, D. K. Basnet and R. K. Nath, \textit{On commuting probability of finite rings}, Indag. Math. \textbf{28}(2) (2017), 372-382.
	
	\bibitem{erdos}P. Erd\"os and P. Tur\'an, \textit{On some problems of a statistical group-theory. IV},  Acta. Math. Acad. Sci. Hungar. {\bf 19} (1968), 413-435.
	
	\bibitem{emskhani1}M. A. Esmkhani and S. M. Jafarian Amiri, \textit{The probability that the multiplication of two ring elements is zero}, J. Algebra Appl. \textbf{16}(11) (2018), 1850054.
	
	\bibitem{emskhani2}M. A. Esmkhani and S. M. Jafarian Amiri, \textit{Characterization of rings with nullity degree at
		least $1/4$}, J. Algebra Appl. \textbf{18}(4) (2019), 1950076.
	
	\bibitem{huckaba}J. A. Huckaba, \textit{Commutative rings with zero divisors}, Monographs and Textbooks in Pure and Applied Mathematics, \textbf{117}, Marcel Dekker, Inc., New York, 1988.
	
	\bibitem{machale}D. MacHale,  \textit{Commutativity in finite rings}, Amer. Math. Monthly \textbf{83} (1976), 30-32.
	
	\bibitem{ragha}R. Raghavendran, \textit{Finite associative rings}, Compos. Math. \textbf{21} (1969), 195-229.
	
	\bibitem{ruiz} S. M. Ruiz, \textit{An algebraic identity leading to Wilson's theorem*}, Math. Gaz. \textbf{80}(489) (1996), 579-582. 	
	
	\bibitem{sarma} D. Sarma and T. Subedi, \textit{The probability that the product of three elements in a finite ring is zero}, J. Algebra Appl. (2025), 2550349 (9 pages). doi: https://doi.org/10.1142/S0219498825503499.
	
\end{thebibliography}

\end{document}